\documentclass[reqno,12pt]{amsart}
\usepackage{amsfonts}
\usepackage{times}
\usepackage{amssymb, amsmath}
\allowdisplaybreaks

\theoremstyle{plain}
\newtheorem{thm}{Theorem}[section]

\newtheorem{lem}[thm]{Lemma}

\theoremstyle{definition}
\newtheorem{defn}{Definition}[section]
\theoremstyle{remark}
\newtheorem{rem}{Remark}[section]
\numberwithin{equation}{section}

\newcommand{\mbh}{\mathbf{h}}
\newcommand{\mbx}{\mathbf{x}}
\newcommand{\mby}{\mathbf{y}}
\def\bu{\mathbf{u}}
\def\bv{\mathbf{v}}
\def\bw{\mathbf{w}}

\newcommand{\real}{\mathbb{R}}

\DeclareMathOperator{\dv}{div}
\DeclareMathOperator{\curl}{curl}

\title[Vanishing viscosity  for helical
Navier-Stokes equations] {The limit of vanishing viscosity for the incompressible 3D
Navier-Stokes equations with helical symmetry}

\author[Q. Jiu]{Quansen Jiu}
\address[Q. Jiu]{School of Mathematical Sciences, Capital Normal University, Beijing, 100048, P. R. China}
\email{jiuqs@cnu.edu.cn}

\author[M. C. Lopes Filho]{Milton C. Lopes Filho}
\address[M. C. Lopes Filho]{Instituto de Matem\'atica\\
Universidade Federal do Rio de Janeiro\\
Cidade Universit\'aria -- Ilha do Fund\~ao\\
Caixa Postal 68530\\
21941-909 Rio de Janeiro, RJ -- Brasil.}
\email{mlopes@im.ufrj.br}

\author[D. Niu]{Dongjuan Niu}
\address[D. Niu]{School of Mathematical Sciences, Capital Normal University, Beijing,100048, P. R. China}
\email{djniu@cnu.edu.cn}

\author[H. J. Nussenzveig Lopes]{Helena J. Nussenzveig Lopes}
\address[H. J. Nussenzveig Lopes]{Instituto de Matem\'atica\\
Universidade Federal do Rio de Janeiro\\
Cidade Universit\'aria -- Ilha do Fund\~ao\\
Caixa Postal 68530\\
21941-909 Rio de Janeiro, RJ  -- Brasil.}
\email{hlopes@im.ufrj.br}

\begin{document}

\nocite{*}

\begin{abstract}
In this paper, we  are concerned with the vanishing viscosity
problem  for the  three-dimensional Navier-Stokes equations with helical symmetry, in the whole space. We choose viscosity-dependent initial $\bu_0^\nu$ with helical swirl, an analogue of the swirl component of axisymmetric flow, of magnitude $\mathcal{O}(\nu)$ in the $L^2$ norm; we assume $\bu_0^\nu \to \bu_0$ in $H^1$. The new ingredient in our analysis is a decomposition of helical vector fields, through which we obtain the required estimates.

\vspace{.3cm}


\vspace{.3cm}


\vspace{.8cm}

{\it Dedicated to Edriss S. Titi, on the occasion of his $60^{th}$ birthday.}

\end{abstract}

\maketitle

\section{introduction}
The initial-value problem for the three-dimensional incompressible
Navier-Stokes equations with viscosity $\nu>0$ is given by
\begin{align}\label{al}
\left\{
\begin{array}{ll}
\partial_t \mathbf{u}^{\nu}+\mathbf{u}^{\nu}\cdot \nabla \mathbf{u}^{\nu}+\nabla p^{\nu} =\nu
\Delta \mathbf{u}^{\nu} &  ({\rm x},t)\in
\mathbb{R}^3\times (0,\infty), \\[1mm]
\dv\mathbf{u}^{\nu}=0 &  ({\rm x},t)\in
\mathbb{R}^3\times (0,\infty), \\[1mm]
\mathbf{u}^{\nu}(t=0,{\rm x})=\mathbf{u_0}^{\nu} & {\rm x}\in \mathbb{R}^3,
\end{array}
\right.
\end{align}
where $\mathbf{x}=(x,y,z)$, $\mathbf{u}^{\nu}=(u_1^{\nu}, u_2^{\nu}, u_3^{\nu})$ is the velocity and $p^{\nu}$ is the pressure.

Formally, when $\nu=0,$ \eqref{al} becomes the classical
incompressible Euler equations
\begin{align}\label{a12}
\left\{
\begin{array}{ll}
\partial_t \mathbf{u}^0+\mathbf{u}^0\cdot \nabla \mathbf{u}^0+\nabla p^0=0 & ({\rm x},t)\in
\mathbb{R}^3\times (0,\infty),\\[1mm]
\dv \mathbf{u}^0=0 & ({\rm x},t)\in\mathbb{R}^3\times (0,\infty), \\[1mm]
\mathbf{u}^0(t=0,{\rm x})=\mathbf{u}_0 & {\rm x}\in \mathbb{R}^3.
\end{array}
\right.
\end{align}

Global existence of weak solutions and local in time well-posedness of strong solutions for problem
\eqref{al} is due to J. Leray, see \cite{L}. There is a vast literature on existence, uniqueness and
regularity of solutions of \eqref{al}, see \cite{MB,Tem} and references therein. Global existence of
strong solutions and uniqueness of weak solutions remain open.

One direction of investigation has been to study the special case of axisymmetric flows, i.e. viscous
flows which are invariant under rotation around a fixed symmetry axis. In particular, among axisymmetric
flows, one distinguishes the no-swirl case. The axisymmetric velocity has three components, a component
in the direction of the axis of symmetry, a radial component, which is orthogonal to the axis of symmetry,
in any plane that contains it, and the azimuthal component, which points in the direction of the rotation
around the axis. No-swirl means that the azimuthal component of velocity vanishes. Global well-posedness
of strong, axisymmetric, solutions of the Navier-Stokes equations \eqref{al} in the no-swirl case, and in
the swirl case when the domain avoids the symmetry axis,
is due to Ladyzhenskaya, see \cite{Lady}. If the domain contains the symmetry axis, global well-posedness is open,
and singularities may occur, but only on the symmetry axis \cite{CKN}. For blow-up criteria in this case,
see \cite{CL}.

Helical flows are another class of three-dimensional flows with an axis of symmetry. These flows are
invariant under a simultaneous rotation around a symmetry axis and translation along the same axis.
The displacement along the axis after one full turn around the axis is an important parameter of helical
symmetry, which, in this article, is assumed to be of unit length. This class of flows is preserved under both
Navier-Stokes and Euler evolution. The mathematical literature
on helical flows is much less extensive than that of axisymmetric flows, but there is growing recent interest.
Well-posedness of strong solutions to three-dimensional Navier-Stokes with
helical symmetry in bounded domain, was proved by Mahalov,
Titi and Leibovich \cite{MTL} with the initial helical velocity $\mathbf{u}_0^{\nu} \in H^1.$ The key observation in
\cite{MTL} is that the helical flows inherit properties of the two-dimensional flow in the plane, to a greater
extent than axisymmetric flows. Specifically, it is  proved in \cite{MTL} that, for a helical vector field $\mathbf{v}$, the following inequality holds true:
\begin{align}\label{a16+}
\|\mathbf{v}\|_{L^4(\Omega)}\le
C\|\mathbf{v}\|^{\frac12}_{L^2(\Omega)}\|\mathbf{v}\|^{\frac12}_{H^1(\Omega)},
\end{align}
where $C>0$ is a constant and $\Omega=\{(x,y,z)\in \mathbb{R}^3| x^2+y^2<1,
0<z<2\pi\}$  is a cylindrical domain.

In analogy with the notion of swirl in axisymmetric flows, we  define the \textit{helical swirl} of a helical vector $\mathbf{v}$ as
\begin{align}\label{hesw}
\eta:=\mathbf{v}\cdot \boldsymbol{\xi}
\end{align}
with $\boldsymbol{\xi}\equiv(y,-x,1)^T.$
Helical swirl plays an important role in global well-posedness of three-dimensional Euler equations with helical symmetry.
In particular, the helical swirl component satisfies a transport equation  and it is conserved along particle trajectories for Euler flow with
helical symmetry.   Assuming  that the initial velocity field has vanishing helical swirl,  Dutrifoy \cite{D} proved the global existence and
uniqueness of classical solutions of three-dimensional Euler equations with helical symmetry. Ettinger and Titi \cite{ET}
obtained the global well-posedness of strong solutions with the initial vorticity belonging to $L^{\infty}$, which is similar to
Yudovich's well-known result for two-dimensional Euler. Recently, Bronzi, Lopes Filho and Nussenzveig Lopes \cite{BLL} verified  global existence of weak solutions when the initial vorticity belongs to $L^p, p>\frac{4}{3}$ with compact support. Subsequently, Jiu, Li and Niu \cite{JLN} generalized this result to include initial vorticities in $L^1\cap L^p, p>1$. All of the aforementioned results assume the initial data has vanishing helical swirl; the problem of global existence for helical Euler with initial nonzero helical swirl remains open.

In this paper, we intend to focus on the vanishing viscosity problem for three-dimensional Navier-Stokes
equations with helical symmetry in the whole space.
We allow initial data for the Navier-Stokes equations with helical swirl of
magnitude $\mathcal{O}(\nu)$, measured
in $L^2$.  We will see that, for viscous flows, the helical swirl is not conserved along particle
trajectories, and the vanishing of the helical swirl is not preserved under Navier-Stokes evolution.
Controlling the magnitude of the swirl component of velocity is the key aspect of obtaining the vorticity
estimates needed to carry out our analysis.

More precisely, for helical velocity fields $\mathbf{u}^\nu$  the
vorticity has the form of
\begin{align}\label{vorPre}
{\boldsymbol{\omega}^\nu=\curl\mathbf{u}^\nu
=\omega^\nu_3\boldsymbol{\xi}+
\left(\frac{\partial \eta^\nu}{\partial y},-\frac{\partial \eta^\nu}{\partial x},0\right)}, \end{align}
where $\omega_3^\nu=\partial_x u^\nu_2-\partial_y
u^\nu_1$ is the third component of the vorticity and
$\eta^\nu=\mathbf{u}^\nu\cdot\boldsymbol{\xi}$ is the  helical swirl.  The equation for
$\omega_3^{\nu}$ can be written as
\begin{align}\label{voreqs}
\partial_t \omega_3^{\nu}+(\mathbf{u}^{\nu}\cdot \nabla)\omega_3^{\nu}+
{\partial_x \eta^{\nu}\partial_y u^{\nu}_3-\partial_y
\eta^{\nu}\partial_x u_3^{\nu}} =\nu \Delta\omega_3^{\nu}.
\end{align}
Clearly, vortex stretching  terms appear in the above equations
( see the third and forth terms on the left hand side) and we cannot
control them uniformly with respect to the viscosity $\nu$. To
overcome this difficulty, we introduce a decomposition of
helical vector fields  to obtain the desired {\em a priori} estimates
(see \eqref{dec}, Lemma \ref{lem2.1}, and Section 4 for more details). Before we investigate
the convergence of the Navier-Stokes
equations to the Euler equations, we   prove  global existence of
weak, and of strong, helical solutions to the Navier-Stokes equations \eqref{al}
 provided that the initial velocity is helical and belongs
to $L^2$ and $H^1$, respectively. This result  is not
included in the existence result of  \cite{MTL} because our fluid domain is the whole space.

This paper is organized as follows. In Section 2 we recall  some useful facts
about helical flows and state our main result. In Section 3 we
present global existence of weak, and strong, solutions to
the three-dimensional helical Navier-Stokes equations in full space, with $L^2$, and $H^1$  initial velocity, respectively.
The key {\em a priori} estimates and the proof of our main result
will be given in Section 4.

\section{Preliminaries and main result}

We begin this section by recalling basic definitions, taken from \cite{ET}, regarding helical symmetry.
Denote by $R_{\theta}$ the rotation  by
an angle $\theta$ around the $z$-axis:
\begin{align}\label{a}
R_{\theta}= \left(
\begin{array}{ccc}
\cos\theta & \sin\theta & 0\\
-\sin\theta& \cos\theta& 0\\
0& 0& 1
\end{array}
\right).
\end{align}

The helical symmetry group $G^{\kappa}$ is a one-parameter
group of isometries of $\mathbb{R}^3$ given by
\begin{align}\label{a13}
G^{\kappa}=\{S_{\theta}:\mathbb{R}^3\longrightarrow
\mathbb{R}^3 \; | \; \theta\in \mathbb{R}\},
\end{align}
where
\begin{align}\label{a14}
S_{\theta}(\mbx)=R_{\theta}(\mbx)+ \left(
\begin{array}{l}
0\\
0\\
\kappa\theta
\end{array}
\right) \equiv \left(
\begin{array}{c}
x\cos \theta+y\sin\theta\\
-x\sin\theta+y\cos\theta\\
z+\kappa \theta
\end{array}
\right),
\end{align}
for $\mbx=(x,y,z)$.
Above, $\kappa$ is a fixed nonzero constant length scale. The transformation
$S_{\theta}$ corresponds to the superposition of a simultaneous rotation around the ${z}$-axis
and a translation along the same ${z}$-axis. A
{\it scalar function} $f:\mathbb{R}^3\longrightarrow \mathbb{R}$ is said to be
helical if
\begin{align}\label{a15}
f(S_{\theta}(\mbx))=f(\mbx), \ \forall \theta\in \mathbb{R}.
\end{align}
A {\it vector field} $\mathbf{v}:\mathbb{R}^3\longrightarrow
\mathbb{R}^3$ is said to be helical, if
\begin{align}\label{a16}
\mathbf{v}(S_{\theta}(\mbx))=R_{\theta}\mathbf{v}(\mbx), \ \forall \theta\in
\mathbb{R}.
\end{align}
Clearly, helical functions and helical vector fields are
periodic in the $z$ direction, with period $2\pi\kappa$.

For simplicity, we will henceforth assume that $\kappa=1$.
By virtue of  the periodicity of helical functions
with respect to the third variable $z$, it is enough to work in the fundamental domain
$\mathcal{D}:=\mathbb{R}^2\times [-\pi,\pi].$  Let
$L^2(\mathcal{D})$ denote the square-integrable functions on $\mathcal{D}$ and let $H^1_{per}(\mathcal{D})$ be the usual $L^2$-based Sobolev
space $H^1$, periodic with respect to $z$, with period $2\pi$; we use the notation $H^2_{per}(\mathcal{D})$ in an analogous manner. We also use
the subscript {\em loc} to denote Sobolev spaces which are local with respect to the horizontal variables $x$ and $y$.

Hereafter we use the notation $c$ and $C$ for generic constants  which are independent of $\nu$.

Below, we state equivalent definitions  of helical functions and
helical vector fields; we refer the reader to
 Claim 2.3 and Claim 2.5 of \cite{ET} for the corresponding proofs.

Set

\begin{equation} \label{xi}
\boldsymbol{\xi} \equiv (y,-x,1)^T.
\end{equation}

\begin{lem}\label{clm1}
A $C^1$ scalar function $f:\mathbb{R}^3
\longrightarrow \mathbb{R}$ is helical if and only if
\begin{align}\label{a17}
y\partial_x f-x\partial_y f+\partial_z f \equiv \boldsymbol{\xi}\cdot\nabla
f=\partial_{\boldsymbol{\xi}}f=0.
\end{align}
\end{lem}

\begin{lem}\label{clm2}
A $C^1$ vector field
$\mathbf{v}=(v_1,v_2,v_3)^T:\mathbb{R}^3 \longrightarrow \mathbb{R}^3$ is
helical if and only if it  the following relations hold true:
\begin{align}
&\partial_{\boldsymbol{\xi}}v_1=v_2, \label{a18}\\
&\partial_{\boldsymbol{\xi}}v_2=-v_1, \label{a19}\\
&\partial_{\boldsymbol{\xi}}v_3=0.\label{a20}
\end{align}
\end{lem}

Next we recall the relation between three-dimensional helical vector fields and their two-dimensional traces on ``slices" $z\,=\,$constant, as discussed in \cite{LMNNT}. Recall that we are assuming $\kappa =1$ so, in the notation of \cite{LMNNT}, $\sigma = 2\pi$.

\begin{lem} \label{twodimpro}
	Set $\mbx=(x_1,x_2,x_3)$. Let $\mathbf{v} = \mathbf{v}(\mbx) $, be a smooth helical vector field and
	let $p=p(\mbx)$ be a smooth helical function. Then there exist {\em unique\/} $\mathbf{w}=(w^1,w^2,w^3)
	=(w^1,w^2,w^3)(y_1,y_2)$ and $q=q(y_1,y_2)$ such that
	\begin{equation} \label{uandpinY}
	\mathbf{v}(\mbx) = R_{ x_3}\mathbf{w}(\mby(\mbx)), \;\;\; p=p(\mbx)=q(\mby(\mbx)),\end{equation}
	with $R_\theta$ given in \eqref{a},
	and
	\begin{equation} \label{Yofx}
	\mby(\mbx) =
	\left[\begin{array}{l}
	y_1 \\ \\
	y_2
	\end{array}\right]
	=
	\left[\begin{array}{cc}
	\cos x_3  & -\sin   x_3  \\ \\
	\sin x_3  & \cos x_3
	\end{array}\right]
	\left[\begin{array}{l}
	x_1\\ \\
	x_2
	\end{array}\right].
	\end{equation}
	
	Conversely, if $\mathbf{v}$ and $p$ are defined through \eqref{uandpinY} for some
	$\mathbf{w}=\mathbf{w}(y_1,y_2)$, $q=q(y_1,y_2)$,
	then $\mathbf{v}$ is a helical vector field and $p$ is a helical scalar function.
\end{lem}
This is precisely Proposition 2.1 in \cite{LMNNT}, in the case $\sigma = 2\pi$, to which we refer the reader for the proof.

Next we will formally introduce the helical swirl, a quantity which plays an
important role in helical flows.

\begin{defn} \label{hesw}
Let $\mathbf{v}$ be a helical vector field.
The helical swirl is defined to be
\[\eta\equiv \mathbf{v}\cdot\boldsymbol{\xi}.\]
\end{defn}

Vorticity, the curl of the velocity field, is a key object in the study of incompressible fluid flow. For helical vector fields, vorticity has a special form.

\begin{lem}\label{lem2.4}
Let $\mathbf{v}$ be a helical vector field. Then its curl, $\boldsymbol{\omega} = \mbox{ curl } \mathbf{v} = (\omega_1, \omega_2, \omega_3)$, is given by
\begin{equation} \label{formforomega}
\boldsymbol{\omega}= \omega_3 \boldsymbol{\xi}+(\partial_y \eta, -\partial_x \eta,0).
\end{equation}
\end{lem}

\begin{proof}
The result follows by a straightforward calculation.
\end{proof}

\begin{rem} \label{omega3}
We note that, in view of Lemma \ref{lem2.4}, if $\mathbf{v}$ is a helical vector field for which the helical swirl vanishes then
\[\mbox{ curl }\mathbf{v}  = (\mbox{ curl }\mathbf{v})_3\boldsymbol{\xi} \equiv (\partial_x v_2 - \partial_y v_1)\boldsymbol{\xi}.\]
\end{rem}

Let $\mathbf{v}$ be a helical vector field. We introduce a decomposition of $\mathbf{v}$ into two other helical vector fields, one of which is orthogonal to the symmetry lines of the helical symmetry group $G^1$. Let $\mathbf{V}$ be defined through the equation
\begin{equation} \label{dec}
\mathbf{v} \equiv \mathbf{V} + \eta \frac{\boldsymbol{\xi} }{|\boldsymbol{\xi}|^2},
\end{equation}
where $\eta$ is the helical swirl introduced in \eqref{hesw}.

\begin{lem}\label{lem2.1}
Let $\boldsymbol{v}$ be a  helical vector field and consider the decomposition \eqref{dec}. Then $\mathbf{V}$ is also a helical vector field. In addition, $\mathbf{V}$ satisfies
$$
\mathbf{V}\cdot \boldsymbol{\xi}=0.
$$
Moreover, if $\mathbf{v}$ is divergence free,  $\mathbf{V}$ is also divergence free.
\end{lem}

\begin{proof}
As $\mathbf{v}$ is helical, we have, thanks to Lemma \ref{clm2},
$\partial_{\boldsymbol{\xi}}\mathbf{v}=(v_2,-v_1,0)^t$. Now, a direct calculation using Lemma \ref{clm2}, together with the expression for $\boldsymbol{\xi}$, yields
\[\partial_{\boldsymbol{\xi}}\eta = \partial_{\boldsymbol{\xi}}(\mathbf{v}\cdot \boldsymbol{\xi}) = (\partial_{\boldsymbol{\xi}}\mathbf{v})\cdot \boldsymbol{\xi} + \mathbf{v} \cdot \partial_{\boldsymbol{\xi}} \boldsymbol{\xi} = 0.\]
Hence, by Lemma \ref{clm1}, $\eta$ is also helical.

Furthermore, we have
\[
\partial_{\boldsymbol{\xi}}\left(\eta \frac{\boldsymbol{\xi}}{|\boldsymbol{\xi}|^2}\right) \]
\[= \frac{\eta}{|\boldsymbol{\xi}|^2}\partial_{\boldsymbol{\xi}} \boldsymbol{\xi} + \boldsymbol{\xi} \partial_{\boldsymbol{\xi}}\left(\frac{\eta}{|\boldsymbol{\xi}|^2}\right) \]
\[= \frac{\eta}{|\boldsymbol{\xi}|^2}\partial_{\boldsymbol{\xi}} \boldsymbol{\xi} = \left( \frac{(\eta\boldsymbol{\xi})_2}{|\boldsymbol{\xi}|^2},-\frac{(\eta\boldsymbol{\xi})_1}{|\boldsymbol{\xi}|^2},0  \right).\]
Therefore, by Lemma \ref{clm2}, it follows that $\eta \boldsymbol{\xi}/|\boldsymbol{\xi}|^2$ is a helical vector field. Consequently, $\mathbf{V}$ is a helical vector field.

In addition, a simple calculation yields
\[\mathbf{V}\cdot \boldsymbol{\xi}=(\mathbf{v}-\eta\frac{\boldsymbol{\xi}}{|\boldsymbol{\xi}|^2})\cdot \boldsymbol{\xi}=0.\]

Finally, suppose that $\mathbf{v}$ is divergence free. Then it follows that
\[\mbox{ div }\left(\eta \frac{\boldsymbol{\xi}}{|\boldsymbol{\xi}|^2}\right)=\partial_{\boldsymbol{\xi}}\left(\frac{\eta}{|\boldsymbol{\xi}|^2}\right) + \frac{\eta}{|\boldsymbol{\xi}|^2}\mbox{ div }\boldsymbol{\xi}=0.\]

Thus we obtain that $\mathbf{V}$ is divergence free as well.

\end{proof}

\begin{rem} \label{Omega3}
Suppose that $\mathbf{v}$ is a helical vector field and let $\mathbf{V}$ be as in \eqref{dec}. Then, since $\mathbf{V}$ is helical and has vanishing helical swirl, it follows that its vorticity, $\mbox{ curl }\mathbf{V} = \boldsymbol{\Omega}$ is given by
\[\boldsymbol{\Omega} = \mbox{ curl } \mathbf{V} =(\partial_x V_2-\partial_y V_1)\boldsymbol{\xi}.\]
See Remark \ref{omega3} for details. Therefore it follows from \eqref{dec}, together with Lemma \ref{lem2.4}, that the third component of the vorticity $\mbox{ curl } \mathbf{v} = \boldsymbol{\omega}$ is given by
\[\omega_3=\Omega_3 +\left( \mbox{ curl } \left( \frac{\eta}{|\boldsymbol{\xi}|^2} \boldsymbol{\xi} \right) \right)_3 \]
\[= \Omega_3 +\partial_x\left(\frac{-\eta x}{|\boldsymbol{\xi}|^2}\right)-\partial_y\left(\frac{\eta y}{|\boldsymbol{\xi}|^2}\right),\]
i.e.
\begin{equation} \label{formforomega3}
\omega_3 = (\partial_x V_2-\partial_y V_1)
-\partial_x\left(\frac{\eta x}{|\boldsymbol{\xi}|^2}\right)-\partial_y\left(\frac{\eta y}{|\boldsymbol{\xi}|^2}\right).
\end{equation}

\end{rem}

We will make use of the following Ladyzhenskaya inequality, valid for helical vector fields, see also \cite{LN}, \cite{Lady} and \cite{MTL}. We give a sketch of the proof for the sake of completeness.

\begin{lem} \label{lem2.2}
There exists a constant $C>0$ such that, for every helical function $\bv\in H^1_{per}(\mathcal{D})$, it holds that
\begin{align} \label{Lady}
\|\bv\|_{L^4(\mathcal{D})}\leq C\|\bv\|_{L^2(\mathcal{D})}^{\frac 1 2}
\|\nabla\bv\|_{L^2(\mathcal{D})}^{\frac 1 2}.
\end{align}
\end{lem}

\begin{proof}[Proof of Lemma \ref{lem2.2}]
Let $\bv \in H^1_{per}(\mathcal{D})$ and consider the vector field $\mathbf{w}$ given in Lemma \ref{twodimpro}, satisfying
\eqref{uandpinY}.
Since $R_{x_3}$ is an orthogonal matrix, we find
$$
|\mathbf{v}(x_1,x_2,x_3)|^2=|\mathbf{w}(y_1,y_2)|^2,
$$
and, hence,
\begin{align}\label{ineq1}
\|\mathbf{w}\|_{L^p(\mathbb{R}^2)}= \frac{1}{\sqrt[p]{2\pi}}\|\mathbf{v}\|_{L^p(\mathbb{R}^2\times [-\pi,\pi])}.
\end{align}
Therefore, using the two-dimensional Ladyzhenskaya inequality (see \cite{Lady}, \cite{S}), we obtain
\begin{align}\label{ineq3}
\|\mathbf{v}\|_{L^4(\mathbb{R}^2\times [-\pi,\pi])}
= \sqrt[4]{2\pi}\|\mathbf{w}\|_{L^4(\mathbb{R}^2)}
\leq c[\|\mathbf{w}\|_{L^2(\mathbb{R}^2)}^{\frac 1 2}\|\nabla_y \mathbf{w}\|_{L^2(\mathbb{R}^2)}^{\frac 1 2}].
\end{align}

Thus, to prove \eqref{Lady}, it suffices to note that, for each $x_3 \in (-\pi,\pi)$, relation \eqref{uandpinY} and \eqref{Yofx} can be inverted, so that
\begin{align}\label{wu}
\mathbf{w}(\mby)= R_{x_3}^{-1}\mathbf{v}(\mbx(\mby)),
\end{align}
with
\begin{align}\label{wu1}
R_{x_3}^{-1}=\left[ {\begin{array}{*{20}c}
	\cos x_3&-\sin x_3&0 \\
	\sin x_3&\cos x_3&0\\
	0&0&1\\
	\end{array} } \right]	
\end{align}
and
\begin{equation} \label{wu2}
\mbx(\mby) =
\left[\begin{array}{l}
x_1 \\ \\
x_2
\end{array}\right]
=
\left[\begin{array}{cc}
\cos x_3 & \sin x_3 \\ \\
-\sin x_3 & \cos x_3
\end{array}\right]
\left[\begin{array}{l}
y_1\\ \\
y_2
\end{array}\right].
\end{equation}
Hence, in view of \eqref{wu}-\eqref{wu2}, it follows that, for some $C>0$,
\begin{align}\label{ineq2}
\|\nabla_{y} \mathbf{w}\|_{L^2(\mathbb{R}^2)}\leq C\|\nabla  \mathbf{v}\|_{L^2(\mathbb{R}^2\times [-\pi,\pi])}.
\end{align}

We conclude by substituting \eqref{ineq1} and \eqref{ineq2} into \eqref{ineq3}.

\end{proof}

Throughout this paper we will make use of the following estimate.

\begin{lem}\label{lem2.3}
Let $\mathbf{v}\in H^1_{per}(\mathcal{D})$ be a  helical vector field. Then
\begin{align}\label{grad}
\|\nabla \mathbf{v}\|_{L^2(\mathcal{D})}
\leq \|\dv \mathbf{v}\|_{L^2(\mathcal{D})}+\|\curl \mathbf{v}\|_{L^2(\mathcal{D})}.
\end{align}
\end{lem}

\begin{proof}
Without loss of generality we may assume that $\mathbf{v}$ is a smooth vector field, compactly supported with respect to $x$ and $y$ , periodic with respect to $z$. The following is a well-known calculus identity:
\begin{equation} \label{calcidentity}
\Delta \mathbf{v}=\nabla (\dv \mathbf{v})-\nabla \times (\curl \mathbf{v}).
\end{equation}

Take the inner product of \eqref{calcidentity} with $-\mathbf{v}$ and integrate in $\mathcal{D}$ to obtain
\[\int_{\mathcal{D}} |\nabla \mathbf{v}|^2 \, dx = \int_{\mathcal{D}} (\dv \mathbf{v})^2 \, dx + \int_{\mathcal{D}} |\curl \mathbf{v}|^2 \, dx.\]

This clearly yields the desired estimate.

\end{proof}

Our objective, in this work, is to show that, under certain assumptions, the vanishing viscosity limit of viscous, helical flows is a helical weak solution of the Euler equations \eqref{a12}; below we give a precise definition of such a weak solution.

\begin{defn} \label{weakEuler}
Fix $T>0$. Let $\bu_0 \in H^1_{per,loc}(\mathcal{D})$. We say $\bu \in C^0(0,T;L^2(\mathcal{D})) \cap L^\infty(0,T; H^1_{per,loc}(\mathcal{D}))$ is a {\em helical weak solution} of the incompressible Euler equations \eqref{a12} with initial velocity $\bu_0$ if the following hold true:
\begin{enumerate}
\item At each time $0\leq t < T$, $\bu(\cdot,t)$ is a helical vector field;
\item For every test vector field $\Phi \in C^\infty_c([0,T)\times\overline{\mathcal{D}})$, periodic in $z$ with period $2\pi$, divergence free, the following identity is valid:
    \[\int_0^T\int_{\mathcal{D}} \Phi_t \cdot \bu + [(\bu \cdot \nabla)\Phi] \cdot \bu \, d\mbx dt + \int_{\mathcal{D}}\Phi_0\cdot \bu_0 \, d\mbx = 0;\]
\item At each time $0\leq t < T$, $\dv \bu(\cdot, t) = $ in the sense of distributions.
\end{enumerate}
\end{defn}
As is usual, it is possible to recover the scalar pressure by means of the Hodge decomposition.

\begin{rem} \label{notwild}
The requirement in Definition \ref{weakEuler} that $\bu \in  L^\infty(0,T; H^1_{per,loc}(\mathcal{D}))$ is not needed to make sense of the terms in the weak formulation. We note, however, that a weak solution as in Definition \ref{weakEuler} satisfies, additionally, a weak form of the inviscid vorticity equation. Definition \ref{weakEuler} excludes, hence, all known examples of {\em wild solutions}.
\end{rem}

We will conclude this section with the statement of our main result.

\begin{thm}\label{thm}
	Let $\{\mathbf{u}_0^{\nu}\}_{\nu>0} \subset H^1_{per}(\mathcal{D})$  be divergence free, helical vector fields and let $\eta^{\nu}_0= \mathbf{u}^{\nu}_0\cdot\boldsymbol{\xi}$ denote their respective helical swirls.
	
	Let $\mathbf{u}_0 \in H^1_{per}(\mathcal{D})$ be a divergence free, helical vector field, such that $\mathbf{u}_0$ has {\it vanishing helical swirl}, i.e., $\mathbf{u}_0 \cdot \boldsymbol{\xi} = 0$.

	Assume that:
	\begin{enumerate}
		\item \[\|\mathbf{u}_0^{\nu} - \mathbf{u}_0\|_{H^1} \to 0 \mbox{\ as\ } \nu \to 0;\]
		\item there exists a constant $C>0$ such that \[\|\eta^{\nu}_0  \|_{L^2(\mathcal{D})} \leq C\nu.\]
	\end{enumerate}
	
	Fix $T>0$. Let $\mathbf{u}^{\nu} \in L^\infty(0,T;H^1_{per}(\mathcal{D}))$ denote the strong solution of the incompressible Navier-Stokes equations \eqref{al} with initial velocity $\mathbf{u}_0^{\nu}$.
	Then, there exists  $\mathbf{u}^0 \in C^0(0,T;L^2(\mathcal{D})) \cap L^\infty(0,T;H^1_{per,loc}(\mathcal{D}))$ such that, passing to subsequences as needed, we have
	\begin{align}\label{b13}
	\mathbf{u}^{\nu} \to \mathbf{u}^0 \mbox{ strongly in  } L^2(0,T; L^2_{loc} (\mathcal{D})),
	\end{align}
	and  $\mathbf{u}^0$ is  a helical weak solution of the incompressible Euler equations, with initial velocity $\bu_0$, and with vanishing helical swirl at any time $0\leq t<T$.
\end{thm}

\section{Global existence of Navier-Stokes equation with helical symmetry}
In this section we discuss well-posedness results for \eqref{al}. In particular, we prove the
global existence of weak helical solutions
 provided  the initial velocity belongs
to $L^2 (\mathcal{D})$ and is helically symmetric, and we prove global existence and uniqueness of strong solutions when the initial data, additionally, belongs to $H^1_{per}$. These results are not included in  \cite{MTL} because our fluid domain is unbounded.

First we introduce a basic mollifier, adapted to the helical symmetry. Let $\rho_1=\rho_1(|\mbx'|)\in C_c^{\infty}(\mathbb{R}^2)$ be a  radially symmetric function satisfying that $\rho_1\geq 0$ and $\int_{\mathbb{R}^2}\rho_1(\mbx')d\mbx'=1$, where $\mbx'=(x,y)$; let also $\rho_2=\rho_2(z)$ in $[-\pi,\pi]$  be a nonnegative, periodic, smooth function with  $\int_{-\pi}^{\pi}\rho_2(z)dz=1$.
Set $J_{\epsilon}\bv$ to be the mollification of a helical vector field $\mathbf{u}\in L^p(\mathcal{D}),
1\leq p\leq \infty,$ given by
\begin{align}\label{d2}
J_{\epsilon}\mathbf{u}=J_{\epsilon}\mathbf{u}(x)\equiv\int_{\mathcal{D}}\rho^{\epsilon}(\mbx-\mby)\mathbf{u}(\mby)d\mby,\
\epsilon>0,
\end{align}
where $\rho^{\epsilon}(\mbx)=\epsilon^{-3}\rho(\frac{\mbx}{\epsilon})$ and $\rho=\rho(\mbx)=\rho_1(\mbx')\rho_2(z)$.

The following lemma provides some basic properties of these mollifiers.
\begin{lem}\label{lem4}
Let $J_{\epsilon}$ be the mollifier defined in \eqref{d2}. Then, for each $\mathbf{u}\in L^p(\mathcal{D})$, $1 \leq  p \leq \infty$,
$J_{\epsilon}\mathbf{u}$ is a $C^{\infty}$ function and
\begin{align}
&(1)\quad J_{\epsilon}[\mathbf{u}(\mby+ \mbh)](\mbx)=J_{\epsilon}[\mathbf{u}(\mby)](\mbx+\mbh),\ \ \ \forall \mbh\in \mathcal{D},
\label{d3}\\
&(2)\quad J_{\epsilon}[\mathbf{u}(R_{\theta}\mby)](\mbx)=J_{\epsilon}[\mathbf{u}(y)](R_{\theta}\mbx),
\ \ \ \hbox{where \ } R_{\theta} \hbox{\ is \ defined \ in \eqref{a}},\label{d4}\\
&(3)\quad J_{\epsilon}[\mathbf{u}(S_{\theta}\mby)](\mbx)=J_{\epsilon}[\mathbf{u}(\mby)](S_{\theta}\mbx),
\ \ \ \hbox{ where \ } S_{\theta} \hbox{\ is \ defined \ in \eqref{a14}},\label{d5}\\
&(4)\quad D^{\alpha}_x(J_{\epsilon}[\mathbf{u}(\mby)])(\mbx)=J_{\epsilon}[D^{\alpha}_y\mathbf{u}(\mby)](\mbx),
\ \ \  |\alpha|\leq m,\ \mathbf{u}\in H^m.\label{d6}
\end{align}
\end{lem}
\begin{proof}[Proof of Lemma \ref{lem4}]
The proof of \eqref{d3} is easily obtained from the definition of $J_\epsilon$,
 \eqref{d2}. Item \eqref{d5} follows directly from
\eqref{d3} and \eqref{d4}, while \eqref{d6} can be found in \cite{MB}.
Item \eqref{d4} follows by a straightforward calculation.
\end{proof}

Let us briefly recall that a weak solution of the Navier-Stokes equations has the regularity $L^\infty(0,T;L^2(\mathcal{D})) \cap L^2(0,T;H^1_{per}(\mathcal{D}))$, whereas a strong solution belongs to $L^\infty(0,T;H^1_{per}(\mathcal{D})) \cap L^2(0,T;H^2_{per}(\mathcal{D}))$.

We can now state and prove a basic result on existence of weak and strong helical solutions to the Navier-Stokes equations \eqref{al}.

\begin{thm}\label{thm2} Fix $\nu > 0$.
Let $\mathbf{u}_0^{\nu}\in L^2(\mathcal{D})$ be a divergence free
and   helical vector field. Fix, also, $T>0$.
\begin{itemize}
\item[(1)] There exists $\mathbf{u}^{\nu} \in L^{\infty}(0,T;L^2 (\mathcal{D})) \cap L^2(0,T;H^1_{per}(\mathcal{D}))$ which is a  helical weak
solution to the three-dimensional Navier-Stokes equations \eqref{al}. In addition, $\mathbf{u}^{\nu}$ satisfies the
following inequality
\begin{align}\label{cc}
\|\mathbf{u}^{\nu}(t)\|_{L^{\infty}(0,T;L^2(\mathcal{D}))}^2+\nu
\|\nabla \mathbf{u}^{\nu}\|_{L^2(0,T;L^2(\mathcal{D}))}^2\leq \|\mathbf{u}_0^{\nu}\|_{L^2(\mathcal{D})}.
\end{align}

\item[(2)] If, in addition, $\mathbf{u}_0^{\nu}\in H^1_{per}(\mathcal{D})$, then
the three-dimensional Navier-Stokes equations \eqref{al} has a unique and
global strong  solution $\mathbf{u}^{\nu} \in
L^\infty(0,T;H^1_{per}(\mathcal{D})) \cap L^2(0,T;H^2_{per}(\mathcal{D}))$ which is helically symmetric.
\end{itemize}
\end{thm}

\begin{proof}[Proof of Theorem \ref{thm2}]
We will begin by establishing (1); the proof will be divided into four steps. As much of this proof is standard, we will be brief.

{\bf Step I} As in \cite{MB}, we construct approximate solutions $\mathbf{u}^{\nu,\epsilon}$
to the  Navier-Stokes equations by solving
\begin{equation} \label{d7}
\left\{
\begin{array}{l}
\mathbf{u}^{\nu,\epsilon}_t+J_{\epsilon}
[(J_{\epsilon}\mathbf{u}^{\nu,\epsilon} \cdot \nabla)(J_{\epsilon}\mathbf{u}^{\nu,\epsilon})]+\nabla p^{\nu,\epsilon}
=\nu J_{\epsilon}(J_{\epsilon}\Delta \mathbf{u}^{\nu,\epsilon}),\\
\dv \mathbf{u}^{\nu,\epsilon}=0,\\
\mathbf{u}^{\nu,\epsilon}(t=0,\mbx)=\mathbf{u}^{\nu,\epsilon}_0,
\end{array}
\right.
\end{equation}
where $\mathbf{u}^{\nu,\epsilon}_0(x):=J_{\epsilon}\mathbf{u}_0^{\nu}(x),$
with $J_{\epsilon}$   defined in \eqref{d2}. By the Picard theorem
(see e.g. \cite{MB}), there exists a unique, global, smooth solution
$\mathbf{u}^{\nu,\epsilon}$ for the regularized Navier-Stokes equations
\eqref{d7}.

{\bf Step II} Next, we show that the approximate solutions
$\mathbf{u}^{\nu,\epsilon}$ preserve helical symmetry.

First we note that $\mathbf{u}_0^{\nu,\epsilon}$ is helical; we use \eqref{d5} in Lemma \ref{lem4} together with the fact that $\mathbf{u}_0^{\nu}$ is
a helical vector field.
We have:
\begin{align}\label{d8}
R_{\theta}^{-1} \mathbf{u}_0^{\nu,\epsilon}(S_{\theta}\mbx)
 &:=R_{\theta}^{-1}J_{\epsilon}\mathbf{u}^{\nu}_0(S_{\theta}\mbx)
=R_{\theta}^{-1}J_{\epsilon}[\mathbf {u}^{\nu}_0(S_{\theta}\mby)](\mbx)\nonumber\\
&:=R_{\theta}^{-1}J_{\epsilon}[R_{\theta}\mathbf{u}^{\nu}_0(\mby)](\mbx)
=J_{\epsilon}\mathbf{u}^{\nu}_0(\mbx):=\mathbf{u}_0^{\nu,\epsilon}(\mbx),
\end{align}

Let $\bar{\mathbf{u}}(\mbx,t)=R^{-1}_{\theta}\mathbf{u}^{\nu,\epsilon}(S_{\theta}\mbx,t)$ and $\bar{p}^{\nu,\epsilon}(\mbx,t)=p^{\nu,\epsilon}(S_{\theta}\mbx,t)$. Direct calculations give that the pair $(\bar{\mathbf{u}},\,\bar{p}^{\nu,\epsilon})$  is a solution of \eqref{d7} with initial data
$\bar{\mathbf{u}}(\mbx,0)=R^{-1}_{\theta}\mathbf{u}_0^{\nu,\epsilon}(S_{\theta}\mbx)\equiv \mathbf{u}_0^{\nu,\epsilon}(\mbx)$.
Hence, by uniqueness of smooth solutions $\mathbf{u}^{\nu, \epsilon}$
of \eqref{d7}, we obtain that
\begin{align}\label{d11}
\mathbf{u}^{\nu,\epsilon}(\mbx,t)\equiv R_{\theta}^{-1} \mathbf{u}^{\nu,\epsilon}(S_{\theta}\mbx,t)
\end{align}
i.e., $\mathbf{u}^{\nu,\epsilon}$ is a helical vector field.

{\bf Step III} In this step we discuss uniform, in $\epsilon$, estimates.
Take the $L^2$-inner production of the regularized momentum equations \eqref{d7} with $\mathbf{u}^{\nu,\epsilon}$ to obtain
\begin{align} \label{diffinequnuep}
\frac{1}{2}\frac{d}{dt}\|\mathbf{u}^{\nu,\epsilon}\|_{L^2(\mathcal{D})}^2+\nu \|\nabla J_{\epsilon}\mathbf{u}^{\nu,\epsilon}
\|_{L^2(0,T;L^2(\mathcal{D}))}^2 \leq 0.
\end{align}

Integrate \eqref{diffinequnuep} in time, from $0$ to $T$, to find
\begin{align} \label{unuepest}
\|\mathbf{u}^{\nu,\epsilon}\|_{L^{\infty}(0,T;L^2(\mathcal{D}))}^2
+\nu \|\nabla J_\epsilon \mathbf{u}^{\nu,\epsilon}\|_{L^2(0,T;L^2(\mathcal{D}))}^2
\leq \|\mathbf{u}_0^{\nu,\epsilon}\|_{L^2(\mathcal{D})}^2.
\end{align}

In view of \eqref{unuepest} it is standard that $\{J_\epsilon\mathbf{u}^{\nu,\epsilon}\}_{\epsilon>0}$ is a compact subset of $L^2(0,T;L^2(\mathcal{D}))$ and hence, passing to subsequences as needed and using properties of mollifiers, we find that $\mathbf{u}^{\nu,\epsilon}$ is a convergent sequence in $L^2(0,T;L^2(\mathcal{D}))$, as $\epsilon \to 0$. We easily obtain that the limit $\mathbf{u}^{\nu}$
satisfies \eqref{al} in the sense of distributions. From the uniform bound in $L^\infty(0,T;L^2(\mathcal{D}))$, \eqref{unuepest}, we obtain that $\mathbf{u}^\nu \in L^\infty(0,T;L^2(\mathcal{D}))$; similarly, we find that $\mathbf{u}^\nu \in L^2(0,T;H^1_{per}(\mathcal{D}))$. Since, from Step II, we deduced that $\mathbf{u}^{\nu,\epsilon}$ is a helical vector field, it follows easily that the limit $\mathbf{u}^{\nu}$ is also helically symmetric. Therefore,
there exists a helical weak solution $\mathbf{u}^{\nu}\in
L^{\infty}(0,T;L^2)\cap L^2(0,T;H^1)$ of \eqref{al}. The energy inequality \eqref{cc} follows by weak convergence in $L^2(0,T;H^1_{per}(\mathcal{D}))$.

{\bf Step IV} Finally, we establish item (2), the existence and uniqueness of a strong solution if the initial data is smoother. From above, we have a weak helical solution in $L^{\infty}(0,T;L^2)\cap L^2(0,T;H^1)$ to the system (\ref{al}). We will show, by energy estimates, that the regularity of $\mathbf{u}^{\nu}$ can be improved to  $L^{\infty}(0,T;H^1)\cap L^2(0,T;H^2)$. Although the estimates below are formal, they can be made rigorous using the regularized equation \eqref{d7} in a
similar way to what was done in Step III.

Taking the $L^2$-inner product of \eqref{al} with $\Delta \mathbf{u}^{\nu}$ we find
\begin{align} \label{enestnablaunu}
&\frac{1}{2}\frac{d}{dt}\|\nabla \mathbf{u}^{\nu}\|_{L^2(\mathcal{D})}^2
+\nu \|\Delta \mathbf{u}^{\nu}\|_{L^2(0,T;L^2(\mathcal{D}))}^2\nonumber\\
&\quad\leq |\int_{\mathcal{D}}
(\mathbf{u}^{\nu}\cdot \nabla) \mathbf{u}^{\nu}\cdot \Delta \mathbf{u}^{\nu}dx|\nonumber\\
&\quad \leq \|\mathbf{u}^{\nu}\|_{L^4}\|\nabla \mathbf{u}^{\nu}\|_{L^4}
\|\Delta \mathbf{u}^{\nu}\|_{L^2}.
\end{align}
Now, since $\bu^\nu$ is a helical vector field, it follows from Lemma \ref{lem2.2} that
\begin{equation}\label{unulady}
\|\bu^\nu\|_{L^4(\mathcal{D})} \leq \|\bu^\nu\|_{L^2(\mathcal{D})}^{\frac 1 2}\|\nabla \bu^\nu\|_{L^2(\mathcal{D})}^{\frac 1 2}.
\end{equation}
Let us examine $\nabla \bu^\nu$. Recall that, from Lemma \ref{lem2.2}, there exists a unique vector field $\bw=\bw(y_1,y_2)$ such that the relation in \eqref{uandpinY} holds true, with $\mby=\mby(\mbx)$ as in \eqref{Yofx}. We write $\nabla \bu^\nu = (\nabla_H \bu^\nu,\partial_{x_3}\bu^\nu)$, where $\nabla_H$ refers to the {\it horizontal} derivatives, i.e. derivatives with respect to $x_1$, $x_2$. In view of \eqref{uandpinY}, \eqref{Yofx} an easy calculation yields, for each $1\leq p \leq \infty$, $m \in \mathbb{N}$, the existence of constants $C_{p,m}$, $c_{p,m} >0$ such that
\begin{equation} \label{cpm}
c_{p,m}\|\nabla^m_y \bw \|_{L^p(\real^2)} \leq \|\nabla^m_H \bu^\nu \|_{L^p(\mathcal{D})} \leq C_{p,m} \|\nabla^m_y \bw \|_{L^p(\real^2)}.
\end{equation}
Since $\nabla\bw$ is a function of two independent variables we may use the two dimensional Ladyzhenskaya inequality for $\nabla \bw$ to find
\[\|\nabla_y\bw\|_{L^4(\real^2)} \leq C\|\nabla_y\bw\|_{L^2(\real^2)}^{\frac 1 2}\|\nabla_y^2 \bw\|_{L^2(\real^2)}^{\frac 1 2} ,\]
from which, together with \eqref{cpm}, it follows that
\begin{equation}\label{nablaHunu}
\|\nabla_H\bu^\nu\|_{L^4(\mathcal{D})} \leq C\|\nabla_H\bu^\nu\|_{L^2(\mathcal{D})}^{\frac 1 2}\|\nabla^2_H \bu^\nu\|_{L^2(\mathcal{D})}^{\frac 1 2}.
\end{equation}
Next, we consider $\partial_{x_3}\bu^\nu$. Recall the criteria in Lemma \ref{clm2} for a vector field to be helical: $\partial_{\boldsymbol{\xi}}v_1 = v_2$,
$\partial_{\boldsymbol{\xi}}v_2 = -v_1$, $\partial_{\boldsymbol{\xi}}v_3 = 0$. Note that $\partial_{x_3}\partial_{\boldsymbol{\xi}} = \partial_{\boldsymbol{\xi}} \partial_{x_3}$. Therefore, since $\bu^\nu$ is a helical vector field, we deduce that
\[\partial_{\boldsymbol{\xi}}\partial_{x_3}u_1^\nu= \partial_{x_3}u_2^\nu; \hspace{.5cm}
\partial_{\boldsymbol{\xi}}\partial_{x_3}u_2^\nu= - \partial_{x_3}u_1^\nu; \hspace{.5cm}
\partial_{\boldsymbol{\xi}}\partial_{x_3}u_3^\nu= 0.
\]
Hence, $\partial_{x_3}\bu^\nu$ is a helical vector field and, therefore, in view of Lemma \ref{lem2.2},
\begin{equation}\label{ux3nulady}
\|\partial_{x_3}\bu^\nu\|_{L^4(\mathcal{D})} \leq C\|\partial_{x_3}\bu^\nu\|_{L^2(\mathcal{D})}^{\frac 1 2}\|\nabla \partial_{x_3}\bu^\nu\|_{L^2(\mathcal{D})}^{\frac 1 2}.
\end{equation}

Notice that both the right-hand-side of \eqref{nablaHunu} and of \eqref{ux3nulady} are bounded by $C\|\nabla\bu^\nu\|_{L^2(\mathcal{D})}^{1/2}\| \Delta\bu^\nu\|_{L^2(\mathcal{D})}^{1/2}$, since, from elliptic regularity theory, we know that all second derivatives are bounded, in $L^2$, by the Laplacian.

We obtain, from \eqref{nablaHunu} and \eqref{ux3nulady},
\begin{equation}\label{nablaunu}
\|\nabla \bu^\nu\|_{L^4(\mathcal{D})} \leq C\|\nabla \bu^\nu\|_{L^2(\mathcal{D})}^{\frac 1 2}\|\Delta \bu^\nu\|_{L^2(\mathcal{D})}^{\frac 1 2}.
\end{equation}

Substituting \eqref{unulady} and \eqref{nablaunu} into \eqref{enestnablaunu} yields
\begin{align}\label{enestnablaunufinal}
\frac{1}{2}\frac{d}{dt}\|\nabla \mathbf{u}^{\nu}\|_{L^2(\mathcal{D})}^2
+\nu \|\Delta \mathbf{u}^{\nu}\|_{L^2(0,T;L^2(\mathcal{D}))}^2 \nonumber \\
\quad \leq \|\mathbf{u}^{\nu}\|_{L^2}^{\frac 1 2}\|\nabla \mathbf{u}^{\nu}\|_{L^2}
\|\Delta \mathbf{u}^{\nu}\|_{L^2}^{\frac 3 2} \nonumber \\
\quad \leq \frac{\nu}{4}||\Delta\bu^\nu\|_{L^2}^2 + C\nu^{-3}\|\bu^\nu\|_{L^2}^2\|\nabla\bu^\nu\|_{L^2}^4,
\end{align}
where we used Young's inequality to obtain the last inequality.
From \eqref{cc} we have
\[\nu\int_0^T \|\nabla\bu^\nu\|_{L^2}^2 dt \leq \|\bu_0^\nu\|_{L^2}^2,\]
so that, by Gronwall's lemma, we obtain
\begin{align}
\|\nabla \mathbf{u}^{\nu}\|_{L^{\infty}(0,T;L^2(\mathcal{D}))}^2\leq \|\nabla \mathbf{u}_0^{\nu}\|_{L^2}^2
\exp\left\{\frac{C\|\mathbf{u}_0^{\nu}\|_{L^2}^4}{\nu^{4}}\right\}.
\end{align}
Thus $\bu^\nu \in L^\infty(0,T;H^1_{per}(\mathcal{D}))$. That $\bu^\nu \in L^2(0,T;H^2_{per})$ follows immediately upon revisiting \eqref{enestnablaunufinal} and integrating in time.

Uniqueness is easily obtained under the regularity of
$\mathbf{u}^{\nu}$. We omit the details.
\end{proof}

\section{Proof of main result }
 We will begin this section by obtaining an evolution equation for the helical swirl. Hereafter we assume that $\bu^\nu_0\in H^1_{per}(\mathcal{D})$ is a divergence free, helical vector field and $\bu^\nu \in L^\infty(0,T;H^1_{per}(\mathcal{D})) \cap L^2(0,T;H^2_{per}(\mathcal{D}))$ is the strong, helically symmetric, solution of \eqref{al} with initial velocity $\bu^\nu_0$, given in Theorem \ref{thm2}. Let $\eta^\nu \equiv \mathbf{u}^\nu \cdot \boldsymbol{\xi}$. Multiply the momentum equation in \eqref{al} by $\boldsymbol{\xi}$  to obtain, after direct calculations,
\begin{equation}\label{bb2}
\left\{
\begin{array}{l}
\partial_t \eta^{\nu}+(\mathbf{u}^{\nu}\cdot \nabla)\eta^{\nu}=\nu \Delta \eta^{\nu}
+2\nu (\partial_x u_2^{\nu}-\partial_y u_1^{\nu}),
\\[2mm]
\eta^{\nu}(t=0,{\rm x})=\eta_0^{\nu}.
\end{array}
\right.
\end{equation}
Clearly, in the case of the Euler equations ($\nu=0$),   the helical swirl $\eta^0:=\mathbf{u^0}\cdot
\boldsymbol{\xi}$  satisfies a transport equation and is conserved along particle paths. This is not the case if $\nu >0$.

Nevertheless, we may still obtain a uniform bound, with respect to $\nu$, for the helical swirl $\eta^{\nu}$.

\begin{lem}\label{lem1}
Fix $T>0$. Let $\mathbf{u}_0^{\nu}\in H^1_{per}(\mathcal{D})$ and $\eta^\nu = \bu_0^\nu\cdot\boldsymbol{\xi}$.
Then there exists a constant $c=c(T)>0,$ independent of $\nu$, such that
\begin{align}\label{b12}
\|\eta^{\nu}\|_{L^{\infty}(0,T;L^2(\mathcal{D}))}+
\sqrt{\nu} \|\nabla \eta^{\nu}\|_{L^2(0,T;L^2(\mathcal{D}))}
\leq c(\|\eta_0^\nu\|_{L^2(\mathcal{D})} + \sqrt{\nu}\|\bu_0^\nu\|_{L^2(\mathcal{D})}).
\end{align}
\end{lem}

\begin{proof}[Proof of Lemma \ref{lem1}]
Multiply both sides of \eqref{bb2} by $\eta^{\nu}$,
integrate the resulting equation in $\mathcal{D}$ and
use that $\dv \mathbf{u}^\nu = 0$ to obtain that
\begin{align}\label{b14}
\frac{1}{2}\frac{d}{dt}\|\eta^{\nu}\|_{L^2(\mathcal{D})}^2
+\nu \|\nabla \eta^{\nu}\|_{L^2(\mathcal{D})}^2
&\leq 2\nu \|\mathbf{u}^{\nu}\|_{L^2(\mathcal{D})}\|\nabla\eta^{\nu}\|_{L^2(\mathcal{D})}\nonumber\\
& \leq \frac{\nu}{2}\|\nabla\eta^{\nu}\|_{L^2(\mathcal{D})}^2
+C\nu\|\mathbf{u}^{\nu}\|_{L^2(\mathcal{D})}^2.
\end{align}
It follows from integration over the time from $0$ to $T$, together with inequality \eqref{cc},
that
\begin{align*}
\|\eta^{\nu}\|_{L^{\infty}(0,T;L^2(\mathcal{D}))}^2+\nu\|\nabla
\eta^{\nu}\|_{L^2(0,T;L^2(\mathcal{D}))}^2 \leq C(\|\eta^{\nu}_0\|_{L^2(\mathcal{D})}^2
+T\nu\|\bu_0^{\nu}\|_{L^2(\mathcal{D})}^2).
\end{align*}

Clearly, this concludes the proof.

\end{proof}

Using the decomposition \eqref{dec}, we introduce
\begin{align}\label{dec1}
\mathbf{U}^{\nu}\equiv \mathbf{u}^{\nu}-\eta^{\nu}\frac{\boldsymbol{\xi}}{|\boldsymbol{\xi}|^2}, \mbox{ and } \boldsymbol{\Omega}^\nu \equiv \curl \mathbf{U}^\nu.
\end{align}
Then  $\mathbf{U}^\nu\cdot \boldsymbol{\xi}=0$ and  $\mathbf{U}^\nu$ is helical.
As noted in  Remark 2.1, we have
\begin{equation}\label{4.2}
\partial_x u_2^{\nu}-\partial_y
u_1^{\nu}=\Omega_3^\nu+\partial_x\left(\frac{-x\eta^\nu}{|\boldsymbol{\xi}|^2}\right)-\partial_y\left(\frac{y\eta^\nu}{|\boldsymbol{\xi}|^2}\right).
\end{equation}
 Moreover, direct calculations give
\begin{align}
\left\{
\begin{array}{l}
\partial_t \mathbf{U}^{\nu}+\mathbf{U}^{\nu}\cdot \nabla \mathbf{U}^{\nu}+\nabla p^{\nu}-\nu \Delta \mathbf{U}^{\nu}\\ \\
\quad =-\displaystyle\frac{\eta^{\nu}}{|\boldsymbol{\xi}|^2}\partial_{\boldsymbol{\xi}}\mathbf{U}^{\nu}-\mathbf{U}^{\nu}\cdot \nabla \left(\displaystyle\frac{\boldsymbol{\xi}}{|\boldsymbol{\xi}|^2}\right)
\eta^{\nu}-\displaystyle\frac{(\eta^{\nu})^2}{|\boldsymbol{\xi}|^2}\partial_{\boldsymbol{\xi}}\left(\displaystyle\frac{\boldsymbol{\xi}}{|\boldsymbol{\xi}|^2}\right) \\ \\
\qquad +2\nu \nabla \eta^{\nu}\cdot
\nabla\left(\displaystyle\frac{\boldsymbol{\xi}}{|\boldsymbol{\xi}|^2}\right)
+\nu \eta^{\nu}\Delta \left(\displaystyle\frac{\boldsymbol{\xi}}{|\boldsymbol{\xi}|^2}\right)
-2\nu \Omega_3^{\nu} \displaystyle\frac{\boldsymbol{\xi}}{|\boldsymbol{\xi}|^2}\\ \\
\qquad\quad -2\nu \left[\curl\left(\displaystyle\frac{\eta^{\nu}\boldsymbol{\xi}}{|\boldsymbol{\xi}|^2}\right)\right]_3
\displaystyle\frac{\boldsymbol{\xi}}{|\boldsymbol{\xi}|^2},\\ \\
\dv \mathbf{U}^{\nu}=0.
\end{array}
\right.
\end{align}

 By Lemma \ref{lem2.4}, $\boldsymbol{\Omega}^{\nu} \equiv \Omega_3^{\nu}\boldsymbol{\xi}$, where
$\Omega_3^{\nu}=\partial_x U_2^{\nu}-\partial_y U_1^{\nu}$.
Direct calculation leads to the following equation for $\Omega_3^{\nu}$:
\begin{align}\label{e9}
\left\{
\begin{array}{l}
\partial_t \Omega_3^{\nu}+\mathbf{U}^{\nu}\cdot \nabla\Omega^{\nu}_3-\nu\Delta\Omega_3^{\nu}=\\ \\
\quad -2\left[\partial_x\left(\displaystyle\frac{\eta^{\nu}(x^2U_1^{\nu}+xyU_2^{\nu})}{|\boldsymbol{\xi}|^4}\right)+
\partial_y\left(\displaystyle\frac{\eta^{\nu}(xyU_1^{\nu}+y^2U_2^{\nu})}{|\boldsymbol{\xi}|^4}\right)\right]\\ \\
\quad
+2\left[
\partial_x\left(\displaystyle\frac{\eta^{\nu}U_1^{\nu}}{|\boldsymbol{\xi}|^4}\right)
+
\partial_y \left(\displaystyle\frac{\eta^{\nu}U_2^{\nu}}{|\boldsymbol{\xi}|^4}\right)
\right]
-\partial_z \left(\displaystyle\frac{(\eta^{\nu})^2}{|\boldsymbol{\xi}|^4}\right)
\\ \\
\quad
-2\nu \left[\partial_x\left(\displaystyle\frac{\partial_x \eta^{\nu}}{|\boldsymbol{\xi}|^2}\right)
 +\partial_y \left(\displaystyle\frac{\partial_y \eta^{\nu}}{|\boldsymbol{\xi}|^2}\right)
 \right]
 \\ \\
\quad
+2\nu\left[\partial_x\left(\displaystyle\frac{x^2\partial_x \eta^{\nu}+xy\partial_y \eta^{\nu}}{|\boldsymbol{\xi}|^4}\right)+
\partial_y \left(\displaystyle\frac{xy\partial_x \eta^{\nu}+y^2\partial_y \eta^{\nu}}{|\boldsymbol{\xi}|^4}\right)
\right]\\ \\
\quad +4\nu \left[\partial_x \left(\displaystyle\frac{x\eta^{\nu}}{|\boldsymbol{\xi}|^6}\right)
+\partial_y \left(\displaystyle\frac{y\eta^{\nu}}{|\boldsymbol{\xi}|^6}\right)
\right]
+2\nu \left[\partial_x \left(\Omega_3^{\nu}\displaystyle\frac{x}{|\boldsymbol{\xi}|^2}\right)
+\partial_y \left(\Omega_3^{\nu}\displaystyle\frac{y}{|\boldsymbol{\xi}|^2}\right)
\right],\\ \\
\Omega_3^{\nu}(t=0,x)=\Omega_{3,0}^{\nu}.
\end{array}
\right.
\end{align}

The following  is a key estimate which will be used to obtain the
compactness of the family of solutions to the Navier-Stokes
equations \eqref{al}, $\nu >0$.

\begin{lem}\label{lem2}
Fix $T>0$. Let $\nu \leq 1$. Assume that $\mathbf{u}_0^{\nu}\in H^1_{per}(\mathcal{D}),$ and
$\eta^{\nu}_0\in L^2(\mathcal{D})$ with
$\|\eta_0^{\nu}\|_{L^2}\le c\nu$. Then, there exists $c=c(T,\|\bu_0^\nu \|_{H^1_{per}(\mathcal{D})})>0$ such that
\begin{align}\label{e11}
\|\Omega_3^{\nu}\|_{L^{\infty}(0,T; L^2)}\leq c.
\end{align}
Furthermore,
\begin{align}\label{RevHeliSwi}
\|\eta^{\nu}\|_{L^{\infty}(0,T;L^2(\mathcal{D}))}+\sqrt{\nu}\|\nabla
\eta^{\nu}\|_{L^2(0,T;L^2(\mathcal{D}))}\leq C\nu,
\end{align}
for some constant $C = C(T,\|\bu^\nu_0\|_{H^1_{per}(\mathcal{D})})>0$ which is independent of $\nu$
\end{lem}

\begin{proof}[Proof of Lemma \ref{lem2}]
Let $t\in [0,T)$ and set
\begin{align}
Y=Y(t):=\int_0^t\|\Omega_3^{\nu}\|_{L^2}^2.
\end{align}	
Then $$Y^{\prime}(t)=\|\Omega_3^{\nu}\|_{L^2}^2,\ \ Y^{\prime\prime}(t)=\frac{d}{dt}\|\Omega_3^{\nu}\|_{L^2}^2.$$

We claim that
\begin{align}\label{e13}
\|\eta^{\nu}\|_{L^{\infty}(0,t;L^2(\mathcal{D}))}+\sqrt{\nu}\|\nabla
\eta^{\nu}\|_{L^2(0,t;L^2(\mathcal{D}))}\leq C\nu (1+Y(t))^{\frac 12},
\end{align}
where $C$ depends on $T$ but is independent of $\nu.$

Indeed, as in the proof of Lemma \ref{lem1}, we multiply the both sides of \eqref{bb2} by $\eta^{\nu}$, integrate the resulting equation in $\mathcal{D}$, and
use the divergence free condition, to obtain
\begin{align}\label{b14}
&\frac{1}{2}\frac{d}{dt}\|\eta^{\nu}\|_{L^2(\mathcal{D})}^2
+\nu \|\nabla \eta^{\nu}\|_{L^2(\mathcal{D})}^2\nonumber\\
&\qquad\leq 2\nu \|\partial_x u_2^{\nu}-\partial_y u_1^{\nu}\|_{L^2(\mathcal{D})}\|\eta^{\nu}\|_{L^2(\mathcal{D})}\nonumber\\
&\qquad\leq 2\nu (\|\Omega_3^{\nu}\|_{L^2(\mathcal{D})}+\|\nabla \eta^{\nu}\|_{L^2(\mathcal{D})}+\|\eta^{\nu}\|_{L^2(\mathcal{D})})\|\eta^{\nu}\|_{L^2(\mathcal{D})}\nonumber\\
&\qquad \leq C\nu^2\|\Omega_3^\nu\|_{L^2(\mathcal{D})}^2 + \frac{\nu}{2}\|\nabla\eta^\nu\|_{L^2(\mathcal{D})}^2 + C\|\eta^\nu\|_{L^2(\mathcal{D})}^2.
\end{align}
where we have used identity \eqref{4.2} and Young's inequality. This gives the estimate
\begin{equation} \label{intermed}
\frac{1}{2}\frac{d}{dt}\|\eta^{\nu}\|_{L^2(\mathcal{D})}^2 +\frac{\nu}{2} \|\nabla \eta^{\nu}\|_{L^2(\mathcal{D})}^2
\leq C\nu^2\|\Omega_3^\nu\|_{L^2(\mathcal{D})}^2 + C\|\eta^\nu\|_{L^2(\mathcal{D})}^2.
\end{equation}

It follows from Gronwall's lemma, upon performing parabolic regularity estimates, that
\[
\|\eta^{\nu}\|_{L^{\infty}(0,t;L^2(\mathcal{D}))}^2+\nu\|\nabla
\eta^{\nu}\|_{L^2(0,t;L^2(\mathcal{D}))}^2
\leq C(\|\eta^{\nu}_0\|_{L^2(\mathcal{D})}^2+\nu^2 Y(t)),\]
for some constant $C=C(T)>0$.
Finally, condition (2) in Theorem \ref{thm}
yields that
\begin{equation}\label{c4}
\|\eta^{\nu}\|_{L^{\infty}(0,t;L^2(\mathcal{D}))}^2+\nu\|\nabla
\eta^{\nu}\|_{L^2(0,t;L^2(\mathcal{D}))}^2   \leq C\nu^2 (1+Y(t))
\end{equation}
i.e.,
\begin{align}\label{cc4}
\|\eta^{\nu}\|_{L^{\infty}(0,t;L^2(\mathcal{D}))}+\sqrt{\nu}\|\nabla
\eta^{\nu}\|_{L^2(0,t;L^2(\mathcal{D}))} \leq C\nu (1+Y(t))^{\frac 12},
\end{align}
where $C=C(T)>0$ is a constant which is independent of $\nu.$ We have established \eqref{e13}.

Next, we use \eqref{e13} to derive estimate \eqref{e11} for $\Omega_3^{\nu}.$ Multiplying both sides of \eqref{e9} by
$\Omega_3^{\nu}$   and integrating in
$\mathcal{D}$, gives
\begin{align*}
&\frac{1}{2}\frac{d}{dt}\|\Omega_3^{\nu}\|_{L^2}^2
+\nu \|\nabla \Omega_3^{\nu}\|_{L^2}^2\\
&\ \ \ \leq 4\int_{\mathcal{D}}\left|\nabla \Omega_3^{\nu} \frac{\eta^{\nu}\mathbf{U}^{\nu}}{|\boldsymbol{\xi}|^4}\right| d\mathbf{x}
+2\int_{\mathcal{D}}\left|\nabla \Omega_3^{\nu}
\frac{\eta^{\nu} \mathbf{U}^{\nu}}{|\boldsymbol{\xi}|^2}\right|d\mathbf{x}
+\int_{\mathcal{D}} \left|\nabla \Omega_3^{\nu}
\frac{|\eta^{\nu}|^2}{|\boldsymbol{\xi}|^4}\right| d\mathbf{x}\\
&\ \ \ +6\nu\int_{\mathcal{D}}\left|\nabla \Omega_3^{\nu} \frac{\nabla \eta^{\nu}}{|\boldsymbol{\xi}|^2}\right|d\mathbf{x}
+8\nu \int_{\mathcal{D}}\left|\nabla \Omega_3^{\nu} \frac{\eta^{\nu}}{|\boldsymbol{\xi}|^4}\right|d\mathbf{x}
+4\nu \int_{\mathcal{D}} |\Omega_3^{\nu}|^2 d\mathbf{x}.
\end{align*}
Then, using Cauchy's inequality together with Young's inequality leads to
\begin{align}
&\frac{d}{dt}\|\Omega_3^{\nu}\|_{L^2}^2+\frac{\nu}{2} \|\nabla \Omega_3^{\nu}\|_{L^2}^2\nonumber\\
&\qquad \leq \frac{c}{\nu}
\left\|\frac{\eta^{\nu} \mathbf{U}^{\nu}}{|\boldsymbol{\xi}|^4}\right\|_{L^2}^2+
\frac{c}{\nu}
\left\|\frac{\eta^{\nu} \mathbf{U}^{\nu}}{|\boldsymbol{\xi}|^2}\right\|_{L^2}^2
+\frac{c}{\nu}\left\|\frac{(\eta^{\nu})^2}{|\boldsymbol{\xi}|^4}\right\|_{L^2}^2\nonumber\\
&\qquad \quad +c\nu \left\|\frac{\nabla \eta^{\nu}}{|\boldsymbol{\xi}|^2}\right\|_{L^2}^2+
c \nu
\left\|\frac{\eta^{\nu}}{|\boldsymbol{\xi}|^4}\right\|_{L^2}^2+c\nu\|\Omega_3^{\nu}\|_{L^2}^2.
\label{b3}
\end{align}
From Lemma \ref{lem2.2}, together with H$\rm\ddot{o}$lder's inequality and \eqref{e13}, it
follows that, for any $\alpha>1$,
\begin{align}
\left\|\frac{\eta^{\nu} \mathbf{U}^{\nu}}{|\boldsymbol{\xi}|^{\alpha}}\right\|_{L^2}^2
& \leq \|\eta^{\nu}\|_{L^4}^2\left\|\frac{\mathbf{U}^{\nu}}{|\boldsymbol{\xi}|^{\alpha}}\right\|_{L^4}^2
\nonumber\\
&\leq \|\eta^{\nu}\|_{L^2}\|\nabla \eta^{\nu}\|_{L^2}\left\|\frac{\mathbf{U}^{\nu}}{|\boldsymbol{\xi}|^{\alpha}}\right\|_{L^2}
\left\|\nabla (\frac{\mathbf{U}^{\nu}}{|\boldsymbol{\xi}|^{\alpha}})\right\|_{L^2}\label{b1}\\
&\leq C\nu (1+Y(t))^{\frac 12} \|\nabla \eta^{\nu}\|_{L^2} \|\mathbf{U}^{\nu}\|_{L^2} \left\|\nabla(\frac{\mathbf{U}^{\nu}}{|\boldsymbol{\xi}|^{\alpha}})\right\|_{L^2}.\nonumber
\end{align}
Using the result in Lemma \ref{lem2.3}, we find
\begin{align}\label{b2}
\left\|\nabla \left(\frac{\mathbf{U}^{\nu}}{|\boldsymbol{\xi}|^{\alpha}}\right)\right\|_{L^2}
&\leq \left(\left\|\curl\left(\frac{\mathbf{U}^{\nu}}{|\boldsymbol{\xi}|^{\alpha}}\right)\right\|_{L^2}
+\left\|\dv\left(\frac{\mathbf{U}^{\nu}}{|\boldsymbol{\xi}|^{\alpha}}\right)\right\|_{L^2}\right)
\nonumber\\
\nonumber\\
&\leq \left\|\frac{\Omega_3^{\nu}}{|\boldsymbol{\xi}|^{\alpha-1}}\right\|_{L^2}+
\alpha\left\|\frac{\boldsymbol{\xi}\times
\mathbf{U}^{\nu}}{|\boldsymbol{\xi}|^{\alpha+2}}\right\|_{L^2}+\alpha
\left\|\frac{\mathbf{U}^{\nu}\cdot\boldsymbol{\xi}}{|\boldsymbol{\xi}|^{\alpha+2}}\right\|_{L^2}
\nonumber\\[3mm]
&\leq \|\Omega_3^{\nu}\|_{L^2}+(1+\alpha)\|\mathbf{U}^{\nu}\|_{L^2}.
\end{align}
Substituting  \eqref{b2} into \eqref{b1} together with the fact that $$\|\boldsymbol{U}^{\nu}\|_{L^{\infty}(0,T;L^2)}\leq \|\bu^{\nu}\|_{L^{\infty}(0,T;L^2)}+\|\eta^{\nu}\|_{L^{\infty}(0,T;L^2)}\leq \|\bu_0^{\nu}\|_{L^2}+c$$ from
\eqref{cc}, \eqref{dec1}, and \eqref{b12}, we have
\begin{align}\label{b4}
\left\|\frac{\eta^{\nu}\mathbf{U}^{\nu}}{|\boldsymbol{\xi}|^{\alpha}}\right\|_{L^2}^2\leq
c\nu\|\nabla \eta^{\nu}\|_{L^2} (1+Y(t))^{\frac 12}(1+\|\Omega_3^{\nu}\|_{L^2}),
\end{align}
where $c$ depends on $\|\bu_0^{\nu}\|_{L^2},\ \|\eta_0^{\nu}\|_{L^2}$ and $T$, independent of $\nu.$

Moreover, noting that
\begin{align}\label{b5}
\left\|\frac{(\eta^{\nu})^2}{|\boldsymbol{\xi}|^4}\right\|_{L^2}\leq \|\eta^{\nu}\|_{L^4}^2
\leq c\|\eta^{\nu}\|_{L^2}\|\nabla \eta^{\nu}\|_{L^2},
\end{align}
we find, by substituting \eqref{b4} and \eqref{b5} into \eqref{b3} and using \eqref{b12} and Young's inequality, that
\begin{align}\label{b7}
&\frac{d}{dt}\|\Omega_3^{\nu}\|_{L^2}^2+\frac{\nu}{2} \|\nabla \Omega_3^{\nu}\|_{L^2}^2\nonumber\\
&\quad \leq
C\|\nabla \eta^{\nu}\|_{L^2}(1+Y(t))^{\frac 12}(1+\|\Omega_3^{\nu}\|_{L^2})
+C\|\nabla \eta^{\nu}\|_{L^2}^2\\
&\quad\qquad\quad +C\nu \|\nabla \eta^{\nu}\|_{L^2}^2+C\nu^2+C\nu \|\Omega_3^{\nu}\|_{L^2}^2\nonumber\\
&\quad\leq C\|\nabla \eta^{\nu}\|_{L^2}^2(1+Y(t))+C(1+\|\Omega_3^{\nu}\|_{L^2}^2)+C\|\Omega_3^{\nu}\|_{L^2}^2+C\|\nabla \eta^{\nu}\|_{L^2}^2, \nonumber
\end{align}
since $\nu \leq 1.$

Recall that
\begin{align}
Y(t)=\int_0^t\|\Omega_3^{\nu}\|_{L^2}^2, Y^{\prime}(t)=\|\Omega_3^{\nu}\|_{L^2}^2,
 Y^{\prime\prime}(t)=\frac{d}{dt}\|\Omega_3^{\nu}\|_{L^2}^2,
\end{align}
so that \eqref{b7} implies that
\begin{align}\label{e18}
Y^{\prime\prime}(t)&\leq C(1+\|\nabla \eta^{\nu}\|_{L^2}^2Y(t))+CY^{\prime}(t)+
C\|\nabla \eta^{\nu}\|_{L^2}^2.
\end{align}
Integrating \eqref{e18} from $0$ to $t$ and using $Y(0)=0$ and $Y'(0)=\|\Omega_{3,0}^{\nu}\|_{L^2}$, we obtain
\begin{align}\label{ccc}
&Y^{\prime}(t)-Y^{\prime}(0) \nonumber\\
&\ \ \leq C\left( t+ Y(t)\int_0^t\|\nabla \eta^{\nu}(s)\|_{L^2}^2\, ds \right)+C\left( Y(t)+\int_0^t\|\nabla \eta^{\nu}\|_{L^2}^2 \, ds \right).
\end{align}
By virtue of \eqref{b12}, it follows that $\|\nabla\eta^{\nu}\|_{L^2(0,t;L^2)}\leq c$, for a constant $c=c(T,\|\eta_0^\nu\|_{L^2(\mathcal{D})}), \|\bu_0^\nu\|_{L^2(\mathcal{D})})>0$. Then \eqref{ccc} becomes
\begin{align}\label{e20}
Y^{\prime}(t)-CY(t)\leq Y^{\prime}(0)+Ct+C.
\end{align}
Consequently,
\begin{align*}
Y(t)\leq C(1+Y^{\prime}(0))e^{-Ct},
\end{align*}
i.e.,
\begin{align}\label{b6}
\|\Omega_3^{\nu}\|_{L^2(0,T;L^2)}\leq C(1+\|\Omega_{3,0}^{\nu}\|_{L^2}),
\end{align}
where $C>0$ depends on $T$.

Combining \eqref{b6} with \eqref{e20}, we  get
\begin{align}\label{e22}
\|\Omega_3^{\nu}\|_{L^{\infty}(0,T;L^2(\mathcal{D}))}
\leq c,
\end{align}
where the constant $c$ only depends on $T,\|\Omega_{3,0}^{\nu}\|_{L^2}$ and $\|\mathbf{u}_0^{\nu}\|_{H^1_{per}(\mathcal{D})},$ but not on $\nu$.
\end{proof}

We are now ready to prove our main result.

\begin{proof}[Proof of Theorem \ref{thm}]
The proof will proceed in three broad steps. First we will show that $\{\bu^{\nu}\}_{\nu>0}$ is a compact subset of $L^2(0,T;L^2(\mathcal{D}))$. Then we will pass to subsequences as needed and show that there is a limit, $\bu^0$, which is in $C(0,T;L^2) \cap L^2(0,T;H^1_{per,loc})$, which is helical, has vanishing helical swirl, and satisfies the weak formulation of the Euler equations. Finally, we will show that $\bu^0 \in L^\infty(0,T;H^1_{per,loc})$.

Recall \eqref{formforomega} and \eqref{formforomega3}. Then,
\begin{align} 
& \dv \bu^\nu = 0 \label{divunu} \\
& \curl \bu^\nu = \left[\Omega_3^\nu - \partial_x\left( \frac{\eta^\nu \,x}{|\boldsymbol{\xi}|^2} \right)
- \partial_y\left( \frac{\eta^\nu \, y}{|\boldsymbol{\xi}|^2}\right) \right]
\boldsymbol{\xi}
+ (\partial_y \eta^\nu, -\partial_x \eta^\nu, 0). \label{curlunu}
\end{align}

Condition (1) of Theorem \ref{thm} implies that $\|\bu_0^\nu\|_{H^1_{per}(\mathcal{D})} \leq C$ and, hence, from the energy inequality in Theorem \ref{thm2}, \eqref{cc}, it follows that $\{\bu^\nu\}_{\nu>0}$ is a  bounded subset of $L^{\infty}(0,T;L^2(\mathcal{D}))$.

From condition (2) of Theorem \ref{thm} together with Lemma \ref{lem2}, \eqref{RevHeliSwi}, we find
\begin{equation} \label{etanugoeszero}
\eta^\nu \to 0 \mbox{ strongly in } L^{\infty}(0,T;L^2(\mathcal{D})),
\end{equation}
and
\begin{equation} \label{gradetanugoeszero}
\nabla \eta^\nu \to 0 \mbox{ strongly in } L^2(0,T;L^2(\mathcal{D})).
\end{equation}

In addition, from Lemma \ref{lem2} we obtained a uniform estimate, with respect to $\nu$, for $\Omega_3^\nu$ in $L^\infty(0,T;L^2)$. Putting these estimates together yields $\curl \bu^\nu$ uniformly bounded in $L^2(0,T;L^2_{loc}(\mathcal{D}))$. (The subscript `loc' is due to the growth of $\boldsymbol{\xi}$ at infinity.) Hence, from Lemma \ref{lem2.3} it follows that $\{\bu^\nu\}_{\nu>0}$ is a bounded subset of $L^2(0,T;H^1_{per,loc} (\mathcal{D}))$.

Therefore, for any bounded sub-domain  $\mathcal{U} \subset \mathcal{D}$, we have that
$\{\mathbf{u}^{\nu}\}_{\nu>0}$ is a bounded subset of
$L^{2}(0,T;H^1(\mathcal{U}))$. In addition, we may use equation \eqref{al} to deduce that $\{\partial_t \mathbf{u}^{\nu}\}_{\nu>0}$ is a bounded subset of $L^2(0,T;H^{-1}(\mathcal{U}))$. It follows from the Aubin-Lions compactness theorem, see \cite{lions}, that $\{\bu^\nu\}_{\nu>0}$ is a compact subset of $L^2(0,T;L^2(\mathcal{U}))$. We may now use a diagonal argument to pass to a subsequence, which we will not relabel, which converges strongly
in $L^2([0,T];L^2_{loc}(\mathcal{D}))$. Passing to a further subsequence if needed, we may assume the convergence is also weak in
$L^2(0,T; H^1_{per,loc}(\mathcal{D}))$.

It is standard that strong convergence in $L^2([0,T];L^2_{loc}(\mathcal{D}))$ is sufficient to show that the limit vector field, denoted $\mathbf{u}^0$, satisfies the weak formulation of the Euler equations in Definition \ref{weakEuler}.

The bounds in $L^2(0,T;H^1_{per,loc} (\mathcal{D}))$, for $\bu^\nu$, and in $L^2(0,T;H^{-1}_{loc}(\mathcal{D})$, for $\partial_t\bu^\nu$, imply that $\{\bu^\nu\}_{\nu>0}$ is a bounded subset of $C^0(0,T;L^2(\mathcal{D}))$. (There is no need to localize this estimate, due to the previous uniform estimate in $L^\infty(0,T;L^2(\mathcal{D}))$.) It follows that $\bu^0 \in C^0(0,T;L^2(\mathcal{D}))$.

It is easy to see that $\bu^0(\cdot, t)$ is a helical vector field, for each $0\leq t < T$ and, also, that $\eta^0 \equiv \bu^0\cdot\boldsymbol{\xi} = 0$.

We have established all conditions of Definition \ref{weakEuler} but one. It remains only to verify that $\bu^0 \in L^\infty(0,T;H^1_{per,loc}(\mathcal{D}))$.

To see this we first note that, in view of \eqref{curlunu},
\[
\Omega_3^\nu
 =
\curl \bu^\nu \cdot \frac{\boldsymbol{\xi}}{|\boldsymbol{\xi}|^2}  +  \partial_x\left( \frac{\eta^\nu \,x}{|\boldsymbol{\xi}|^2} \right)
+ \partial_y\left( \frac{\eta^\nu \, y}{|\boldsymbol{\xi}|^2}\right)
- (\partial_y \eta^\nu, -\partial_x \eta^\nu, 0)\cdot \frac{\boldsymbol{\xi}}{|\boldsymbol{\xi}|^2}.
\]

Each of the terms on the right-hand-side converges, as $\nu \to 0$, weakly in $L^2(0,T;L^2(\mathcal{D}))$ and, in view of \eqref{etanugoeszero} and \eqref{gradetanugoeszero}, the weak limit, in $L^2(0,T;L^2(\mathcal{D})$, of the right-hand-side is $\curl \bu^0 \cdot \frac{\boldsymbol{\xi}}{|\boldsymbol{\xi}|^2} $.

In addition, since $\Omega_3^\nu$ is bounded in $L^\infty(0,T;L^2(\mathcal{D}))$, we may assume, passing to further subsequences as needed, that the convergence of the right-hand-side is also weak-$\ast$ in $L^\infty(0,T;L^2(\mathcal{D}))$, so that
\[\curl \bu^0 \cdot \frac{\boldsymbol{\xi}}{|\boldsymbol{\xi}|^2} \in L^\infty(0,T;L^2(\mathcal{D})).\]

By virtue of $\eta^0=0$ and $\bu^0$ being a helical vector field, we find $\curl \bu^0 \equiv \boldsymbol{\omega}^0 = \omega^0_3 \boldsymbol{\xi}$, see Lemma \ref{lem2.4} and Remark \ref{omega3}.
Therefore we deduce
\[\curl \bu^0 \in L^\infty(0,T;L^2_{loc}(\mathcal{D})),\]
which, together with $\dv \bu^0 = 0$, imply
\[\bu^0 \in L^\infty(0,T;H^1_{per,loc}(\mathcal{D})),\]
as desired.

This completes the proof.

\end{proof}

In this article we have focused on the vanishing viscosity limit for helically symmetric flows. As we have discussed, helically symmetric solutions of the Navier-Stokes equations do not form singularities in finite time, whereas helical Euler is only known to have global solutions if the helical swirl vanishes. Furthermore, vanishing helical swirl is preserved by the Euler evolution, but not by Navier-Stokes. Given these distinctions, it seemed natural to explore the vanishing viscosity problem under helical symmetry. The key issue was to be able to control the helical swirl and to ensure that it vanishes as $\nu \to 0$. 

The relevant problem which still remains open, in this direction, is global existence for helical Euler with nonzero helical swirl.

\vspace{1cm}

\footnotesize{Acknowledgments:}
Quansen Jiu is partially supported by NSFC (No.11671273 and No. 11231006).
The research of M. C. Lopes Filho was supported in part by CNPq Grant \#306886/2014-6.
The research of H. J. Nussenzveig Lopes was partially supported by CNPq Grant \#307918/2014-9, and  FAPERJ Grant \#E-26/202.950/2015. The research of Dongjuan Niu was supported in part by FAPESP Grant \#2008/09473-0; Dongjuan Niu is partially supported by NSFC (No.11471220) and by the Beijing Municipal Commission of Education grants (No. KM201610028001). This material is based upon work supported by the National Science Foundation under Grant No. DMS-1439786 while the second and fourth authors were in residence at the Institute for Computational and Experimental Research in Mathematics (ICERM) in Providence, RI, during the Spring 2017 semester. The authors thank Anna Mazzucato for useful discussions.


\vspace{.1cm}


\begin{thebibliography}{100}

\bibitem{BLL}
A. Bronzi, M.~C. Lopes Filho and H.~J. Nussenzveig Lopes. Global existence of a weak solution of the incompressible Euler equations with helical symmetry and $L^p$ vorticity.
\newblock{\em Indiana Univ. Math. J. } 64, no. 1, 309-341, 2015.

\bibitem{CKN}
L.Caffarelli, R. Kohn, L. Nirenberg. Partial regularity of suitable
weak solutions of the Navier-Stokes equations.
\newblock{\em Comm. Pure Appl. Math.} 35:771-831, 1982.


\bibitem{CL} D. Chae, J. Lee. On the regularity of the axisymmetric
solutions of the Navier-Stokes equations.
\newblock{\em Math. Z.} 23: 645-671, 2002.




\bibitem{D}
A. Dutrifoy.
 Existence globale en temps de solutions h$\acute{e}$lico\"{\i}dales
des $\acute{e}$quations d'Euler.
\newblock{\em C. R. Acad. Sci. Paris S$\acute{e}$r I Math.} 329(7):653-656, 1999.

\bibitem{ET}
B. Ettinger and E.~S. Titi.
 Global existence and uniqueness of weak
solutions of 3D Euler equations with helical symmetric in the
absence of vorticity stretching.
\newblock{\em SIAM J. Math. Anal.} 41(1):269-296, 2009.

\bibitem{lions}
J.~L. Lions.
Quelques methodes de r\'esolution des probl\`emes aux limites non lin\'eaires.
Dunod, Paris, 1969.


\bibitem{LN}
Jitao Liu, Dongjuan Niu. Global well-posedness of three-dimensional Navier–Stokes equations
with partial viscosity under helical symmetry.
\newblock{\em Z. Angew. Math. Phys.} DOI: 10.1007/s00033-016-0645-z,  2017.





\bibitem{JLN}
Q.S. Jiu, Jun Li, Dongjuan Niu.
Global existence of weak solutions to the three-dimensional Euler equations with helical symmetry.
\newblock{\em J. Differential Equations.} 262, 5179–5205, 2017.


\bibitem{Lady}
O. A. Lady\v{z}enskaja. The mathematical theory of viscous
incompressible flow.  Gordon and Breach,  New York, 1969.

\bibitem{L}
J. Leray.
Essai sur le mouvement d'un liquide visqueux emplissant l'espace.
\newblock{\em Acta Math.} 63:193-248, 1934.



\bibitem{LMNNT}
M.~C. Lopes Filho, A. Mazzucato, Dongjuan Niu, H.~J. Nussenzveig Lopes and E.~S. Titi.
Planar limimts of three-dimensional incompressible flows with helical symmetry.
\newblock{\em J. Dyn. Diff. Equat.} 26:843-869, 2014.



\bibitem{MTL}
 A. Mahalov, E.~S. Titi and S. Leibovich.
 Invariant Helical Subspaces for the Navier-Stokes equations.
\newblock{\em Arch. Rational Mech. Anal.} 112:193-222, 1990.

\bibitem{MB}
A.~J. Majda, A.~L. Bertozzi. Vorticity and incompressible flow.
Cambridge University Press, 2002.


\bibitem{S}
H. Sohr.
The Navier-Stokes Equations.
Birkhauser Verlag, 2000.

\bibitem{Tem}
R. Temam. Navier-Stokes equations. North-Holland, Amsterdam, 1977.






\end{thebibliography}
\end{document}